\documentclass{emsprocart}
\usepackage{graphics,graphicx}
 
\contact[agaif@mi.ras.ru]{Alexander A. Gaifullin, Steklov Mathematical Institute of Russian Academy of Science, Gubkina str. 8, Moscow, 119991, Russia,\\
Department of Mechanics and Mathematics, Moscow State University, Leninskie Gory~1, Moscow, 119991, Russia,\\
Institute for Information Transmission Problems of Russian Academy of Science (Kharkevich Institute), Bolshoy Karetny per. 19, build.1, Moscow, 127051, Russia}
 
\numberwithin{equation}{section}

\newtheorem{theorem}{Theorem}[section]
\newtheorem{corollary}[theorem]{Corollary}
\newtheorem{lemma}[theorem]{Lemma}
\newtheorem{proposition}[theorem]{Proposition}
\newtheorem{conjecture}[theorem]{Conjecture}
\newtheorem{problem}[theorem]{Problem}

\theoremstyle{definition}
\newtheorem{definition}[theorem]{Definition}
\newtheorem{remark}[theorem]{Remark}

\newcommand{\Log}{\operatorname{Log}}
\renewcommand{\Re}{\operatorname{Re}}
\newcommand{\dist}{\operatorname{dist}}
\newcommand{\sn}{\operatorname{sn}}
\newcommand{\cn}{\operatorname{cn}}
\newcommand{\dn}{\operatorname{dn}}

\newcommand{\R}{\mathbb{R}}
\newcommand{\C}{\mathbb{C}}
\newcommand{\Z}{\mathbb{Z}}
\newcommand{\Q}{\mathbb{Q}}
\newcommand{\bell}{\boldsymbol{\ell}}
\newcommand{\ba}{\mathbf{a}}
\newcommand{\bb}{\mathbf{b}}
\newcommand{\bc}{\mathbf{c}}
\newcommand{\bd}{\mathbf{d}}
\newcommand{\bm}{\mathbf{m}}
\newcommand{\bn}{\mathbf{n}}
\newcommand{\bs}{\mathbf{s}}
\newcommand{\bx}{\mathbf{x}}
\newcommand{\by}{\mathbf{y}}

\title[Flexible Polyhedra and Their Volumes]{Flexible Polyhedra and Their Volumes}

\author[Alexander A. Gaifullin]{Alexander A. Gaifullin\thanks{The work is partially supported by the Russian Foundation for Basic Research (grants 14-01-00537 and~16-51-55017).}}

\begin{document}

\begin{abstract}
We discuss some recent results on flexible polyhedra and the bellows conjecture, which claims that the volume of any flexible polyhedron is constant during the flexion.  Also, we survey main methods and several open problems in this area.
\end{abstract}



\maketitle

\section{Three-dimensional flexible polyhedra}\label{section_three}

Before speaking on three-dimensional flexible polyhedra let us discuss briefly the two-dimensional case. Consider a  \textit{planar polygonal linkage}, i.\,e., a closed polygonal curve in plane that is allowed to be deformed so that the side lengths remain constant and the angles between consecutive sides change continuously. Equivalently, a planar polygonal linkage can be viewed as a  planar mechanism consisting of bars of fixed lengths connected in a cyclic order by revolving joints. A planar polygonal linkage is called a \textit{flexible polygon\/} if it admits deformations not induced by isometries of plane. 
Notice that we can consider both \textit{embedded} and \textit{self-intersecting} flexible polygons, see Fig.~\ref{fig_polygon}, (a) and~(b), respectively. Obviously, a triangle is rigid, and a generic polygon with at least four sides is flexible. Configuration spaces (or moduli spaces) of planar polygonal linkages were studied extensively by many mathematicians, see~\cite{CoDe04},~\cite{FaSc06}, and references therein.  

\begin{figure}[t]
\unitlength=3.8mm

\begin{center}
\begin{picture}(31,4)

\put(5,-0.5){a}
\put(22.6,-0.5){b}

\put(0,1){%
\begin{picture}(3,3)

\put(0,1.5){\circle*{.3}}
\put(1.5,0){\circle*{.3}}
\put(3,1.5){\circle*{.3}}
\put(1.5,3){\circle*{.3}}

\put(0,1.5){\line(1,1){1.5}}
\put(0,1.5){\line(1,-1){1.5}}
\put(3,1.5){\line(-1,1){1.5}}
\put(3,1.5){\line(-1,-1){1.5}}

\end{picture}%
}

\put(4,2.5){\vector(1,0){2}}
\put(6,2.5){\vector(-1,0){2}}

\put(7,2.5){%
\begin{picture}(3.8,2)%0 - poseredine

\put(0,0){\circle*{.3}}
\put(1.9,0.95){\circle*{.3}}
\put(1.9,-0.95){\circle*{.3}}
\put(3.8,0){\circle*{.3}}

\put(0,0){\line(2,1){1.9}}
\put(0,0){\line(2,-1){1.9}}
\put(3.8,0){\line(-2,1){1.9}}
\put(3.8,0){\line(-2,-1){1.9}}

\end{picture}%
}

\put(15,1){%
\begin{picture}(5,4)

\put(0,0){\circle*{.3}}
\put(1,3){\circle*{.3}}
\put(4,3){\circle*{.3}}
\put(5,0){\circle*{.3}}

\put(0,0){\line(1,3){1}}
\put(0,0){\line(4,3){4}}
\put(5,0){\line(-1,3){1}}
\put(5,0){\line(-4,3){4}}
\end{picture}%
}

\put(20.5,2.5){\vector(1,0){2}}
\put(22.5,2.5){\vector(-1,0){2}}

\put(23,2){%
\begin{picture}(8,1)

\put(0,0){\circle*{.3}}
\put(3,1){\circle*{.3}}
\put(5,1){\circle*{.3}}
\put(8,0){\circle*{.3}}

\put(0,0){\line(3,1){3}}
\put(0,0){\line(5,1){5}}
\put(8,0){\line(-3,1){3}}
\put(8,0){\line(-5,1){5}}

\end{picture}%
}

\end{picture}
\end{center}

\caption{(a) An embedded flexible polygon. (b) A self-intersecting flexible polygon}\label{fig_polygon}
\end{figure}
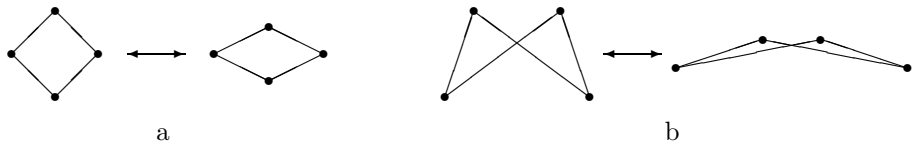

The concept of a flexible polygon can be generalised to higher dimensions in several ways. First, one can study polygonal linkages in Euclidean spaces of dimensions greater than two. For instance, beautiful results on geometry and topology of configuration spaces of polygonal linkages in~$\R^3$ were obtained in~\cite{Kly94}, \cite{KaMi96}, \cite{HaKn98}.  Second, one can study flexions of $k$-dimensional closed polyhedral surfaces in~$\R^n$, where $k<n$. In this paper we shall concern only the most rigid case, that is, the case of an $(n-1)$-dimensional  closed polyhedral surface in~$\R^n$. Besides, we shall consider only oriented polyhedral surfaces. An $(n-1)$-dimensional  oriented closed polyhedral surface~$S$ can be thought of as the boundary of an $n$-dimensional polyhedron in~$\R^n$.  This is completely true if the surface~$S$ is embedded. Nevertheless, if~$S$ is self-intersecting, then it can also be interpreted as the boundary of certain $n$-dimensional object, see Section~\ref{section_BC} for more details. In this paper,  under a \textit{polyhedron} in~$\R^n$ we always mean an $(n-1)$-dimensional  oriented closed polyhedral surface. A \textit{flexion} of a polyhedron is  a continuous deformation of it such that the combinatorial type of the polyhedron does not change under the deformation, each face of the polyhedron remains congruent to itself during the deformation, and the whole polyhedron  does not remain congruent to itself during the deformation.  

In this section we consider flexible polyhedra in three-space. They can be visualised as follows. Assume that faces of an oriented closed polyhedral surface are rigid plates and adjacent faces are connected by hinges at edges. If this mechanism admits nontrivial deformations, then it is a \textit{flexible polyhedron}. In spite of this mechanical point of view, we allow the polyhedral surface to be self-intersecting. However, embedded flexible polyhedra, which are called \textit{flexors}, are of a special interest. Polyhedra that admit no flexions are called \textit{rigid}.

We shall mostly work with simplicial polyhedra. A simplicial polyhedron in three-space is a polyhedron with triangular faces. A strict definition of  a (not necessarily embedded)
simplicial polyhedron is as follows. Let $K$ be an oriented closed two-dimensional simplicial manifold. A \textit{simplicial polyhedron} of combinatorial type~$K$ is a mapping $P\colon K\to\R^3$ whose restriction to every simplex of~$K$ is affine linear.  A polyhedron~$P$ is called \textit{non-degenerate} if   the restriction of~$P$ to every simplex of~$K$ is an embedding, and $K$ cannot be decomposed into the union of several subcomplexes $K_1,\ldots,K_s$ such that $\dim K_i=2$ for all~$i$ and $P(K_i\cap K_j)$ is contained in a line for any $i\ne j$. The latter condition is needed to exclude examples of polyhedra like the polyhedron in Fig.~\ref{fig_2tetr}, which is geometrically the union of two tetrahedra  that have a common edge and can rotate independently around it. In~\cite[Sect.~2]{Gai15a} the author introduced a weaker form (for $s=2$ only) of this condition. However, it is more natural to require this condition in full generality.

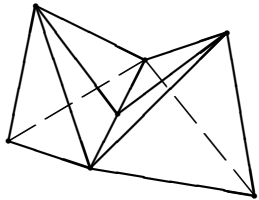
\begin{figure}
\unitlength=3.6mm
\begin{center}
\begin{picture}(9,7)

\put(0,2){\circle*{.2}}
\put(1,7){\circle*{.2}}
\put(3,1){\circle*{.2}}
\put(4,3){\circle*{.2}}
\put(5,5){\circle*{.2}}
\put(9,0){\circle*{.2}}
\put(8,6){\circle*{.2}}

\thinlines

\multiput(0,2)(1.35,0.81){4}{\line(5,3){1}}
\multiput(5,5)(1.08,-1.35){4}{\line(4,-5){.8}}

\thicklines

\put(0,2){\line(1,5){1}}
\put(0,2){\line(3,-1){3}}
\put(1,7){\line(1,-3){2}}
\put(1,7){\line(3,-4){3}}
\put(1,7){\line(2,-1){4}}
\put(3,1){\line(1,2){2}}
\put(3,1){\line(1,1){5}}
\put(3,1){\line(6,-1){6}}
\put(5,5){\line(3,1){3}}
\put(8,6){\line(1,-6){1}}
\put(4,3){\line(4,3){4}}

\end{picture}
\end{center}
\caption{A degenerate simplicial polyhedron}\label{fig_2tetr}
\end{figure}

According to a celebrated theorem of Cauchy~\cite{Cau13}, any convex polyhedron is rigid. Unlike the two-dimensional case, a generic polyhedron of any combinatorial type is rigid. This was proved by Gluck~\cite{Glu75} for three-dimensional polyhedra of topological type of sphere, and by Fogelsanger~\cite{Fog88} for arbitrary polyhedra of any dimension $n\ge 3$.  The first examples of flexible (self-intersecting) polyhedra were obtained by Bricard~\cite{Bri97}. All these flexible polyhedra were of combinatorial type of octahedron. 
Since any trihedral angle in three-space is rigid, it follows easily that all polyhedra with not more than~$5$ vertices are rigid, and all flexible polyhedra with $6$ vertices must have combinatorial type of octahedron. Hence Bricard's flexible octahedra are the simplest flexible polyhedra. (Here and below, polyhedra of combinatorial type of octahedron are  called octahedra.)
Bricard obtained a complete classification of flexible octahedra: He found three types (continuous families) of flexible octahedra, and proved that there are no other flexible octahedra. Bricard's octahedra of the first type and of the second type are symmetric about a line and about a plane, respectively, and Bricard's octahedra of the third type possess no symmetries and are called \textit{skew flexible octahedra}. In addition, Bricard proved that all flexible octahedra are self-intersecting. 

Let us consider in more details Bricard's octahedra of the first type, see Fig.~\ref{fig_Bricard}. The vertices of the octahedron are denoted by $\ba_1,\ba_2,\ba_3,\bb_1,\bb_2,\bb_3$ so that $[\ba_1\bb_1]$,  $[\ba_2\bb_2]$, and  $[\ba_3\bb_3]$ are the diagonals of the octahedron, and all other pairs of vertices are joined by edges of the octahedron. 
The line of symmetry~$l$ is the common perpendicular bisector of the diagonals $[\ba_1\bb_1]$,  $[\ba_2\bb_2]$, and  $[\ba_3\bb_3]$.
Any triple of vertices pairwise joined by edges span a face of the octahedron. In Fig.~\ref{fig_Bricard}, the visible edges, the invisible edges, and the diagonals of the octahedron are indicated by thick full lines, thin full lines, and dotted lines respectively, and the line of symmetry~$l$ is indicated by a dashed line. It is easy to see that Bricard's octahedra of the first type are self-intersecting. To explain why these octahedra  are flexible, we shall prove the following more general assertion.

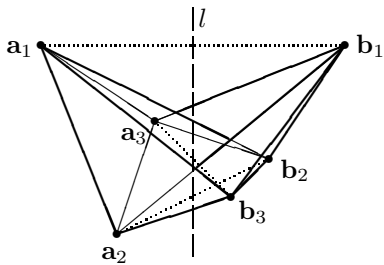
\begin{figure}
\unitlength=5mm
\begin{center}

\begin{picture}(8,6)
\put(4,0){%
\begin{picture}(8,5)
\put(-4,5){\circle*{.2}}
\put(4,5){\circle*{.2}}
\put(-1,3){\circle*{.2}}
\put(1,1){\circle*{.2}}
\put(-2,0){\circle*{.2}}
\put(2,2){\circle*{.2}}

\thinlines

\multiput(0,-.6)(0,.75){9}{\line(0,1){.6}}

\qbezier[65](-4,5)(0,5)(4,5)
\qbezier[31](-2,0)(0,1)(2,2)
\qbezier[21](-1,3)(0,2)(1,1)

\thicklines

\put(-4,5){\line(2,-5){2}}
\put(4,5){\line(-6,-5){3.88}}
\put(-2,0){\line(3,1){3}}
\put(1,1){\line(3,4){3}}
\put(1,1){\line(1,1){1}}
\put(2,2){\line(2,3){2}}
\put(-4,5){\line(2,-1){4.95}}
\put(-4,5){\line(5,-4){5}}
\put(4,5){\line(-5,-2){4.4}}

\thinlines

\put(2,2){\line(-2,1){1}}
\put(-4,5){\line(3,-2){3}}
\put(-2,0){\line(6,5){6}}
\put(4,5){\line(-5,-2){5}}
\put(-2,0){\line(1,3){1}}
\put(-1,3){\line(3,-1){3}}

\put(-4.9,4.8){$\ba_1$}
\put(-2.4,-.65){$\ba_2$}
\put(-1.9,2.5){$\ba_3$}
\put(4.3,4.8){$\bb_1$}
\put(2.3,1.5){$\bb_2$}
\put(1.2,.4){$\bb_3$}
\put(.15,5.5){$l$}
\end{picture}%
}
\end{picture}
\end{center}

\caption{Bricard's octahedron of the first type}\label{fig_Bricard}
\end{figure}

\begin{proposition}
A generic simplicial polyhedron in~$\R^3$ that has topological type of sphere and is symmetric about a line so that no vertex is symmetric to itself and no two vertices connected by an edge are symmetric to each other is flexible.
\end{proposition}

\begin{proof}
Let $x,y,z$ be the standard Cartesian coordinates in~$\R^3$. Let $K$ be a two-dimensional simplicial sphere, and let $\phi$ be a simplicial automorphism of~$K$ such that, for each vertex $v$ of $K$, $\phi(v)\ne v$ and $[v\,\phi(v)]$ is not an edge of~$K$. Consider polyhedra $P\colon K\to\R^3$ symmetric about a line such that $P(\phi(v))$ is symmetric to~$P(v)$ for all~$v$. By an isometry of~$\R^3$, we may achieve that  the line of symmetry of~$P$ is the $z$-axis, and besides, the point~$P(u)$ lies on the $x$-axis,  where $u$ is certain chosen vertex of~$K$. Let $\Theta$ be the set of all polyhedra $P\colon K\to\R^3$ satisfying these conditions.  Our aim is to prove that a generic polyhedron in~$\Theta$ is flexible.

Suppose that $K$ has $2k$ vertices. Let $v_1=u$, $v_2,\ldots,v_k$ be pairwise different vertices of~$K$ such that, for each vertex~$v$, exactly one of the two vertices~$v$ and~$\phi(v)$ is in this list. Suppose that 
$P(u)=(x_1,0,0)$ and $P(v_i)=(x_i,y_i,z_i)$ for $i=2,\ldots,k$. Then $P(\phi(v_i))=(-x_i,-y_i,z_i)$ for all~$i$, where $y_1=z_1=0$. Therefore,
\begin{equation}\label{eq_coord}
x_1,x_2,y_2,z_2,x_3,y_3,z_3,\ldots,x_k,y_k,z_k
\end{equation}
are coordinates in~$\Theta$ that identify~$\Theta$ with $\R^{3k-2}$.

 Since $K$ is homeomorphic to a sphere, Euler's formula implies that $K$ has $6k-6$ edges. No edge of~$K$ is fixed setwise by~$\phi$. Let $\varepsilon_1,\ldots,\varepsilon_{3k-3}$ be representatives of all $3k-3$ pairs $\{\varepsilon, \phi(\varepsilon)\}$, where $\varepsilon$ are edges of~$K$. For every $j=1,\ldots,3k-3$, let~$q_j$ be the square of the length of~$P(\varepsilon_j)$. Then the square of the length of~$P(\phi(\varepsilon_j))$ is also equal to~$q_j$. Each~$q_j$ is a quadratic polynomial in $3k-2$ coordinates~\eqref{eq_coord}. Together these polynomials yield the polynomial mapping 
$$
\mathbf{q}\colon\Theta=\R^{3k-2}\to\R^{3k-3}.
$$
It follows easily from the implicit function theorem that a generic point $P\in\R^{3k-2}$ is contained in a non-constant smooth curve~$P_t$ in~$\R^{3k-2}$ such that $t$ runs over an interval $(-\alpha,\alpha)$, $P_0=P$, and $\mathbf{q}(P_t)=\mathbf{q}(P)$ for all~$t$.  This means that the edge lengths of~$P_t$ are constant as~$t$ varies. Hence $P_t$ is a flexible polyhedron.
\end{proof}

Bricard's octahedra of the second type are symmetric about a plane so that~$\ba_1$ is symmetric to~$\bb_1$, $\ba_2$ is symmetric to~$\bb_2$, and $\ba_3$ and~$\bb_3$ lie on the plane of symmetry. The proof of the flexibility  is completely similar. A geometric description of Bricard's octahedra of the third type is more complicated, and we shall not discuss it here, see~\cite{Bri97}, \cite{Ben12}. Nevertheless, in Section~\ref{section_high} we shall see that  the third type of Bricard's octahedra is the simplest one from the algebraic viewpoint. 

The first flexor (i.\,e., embedded flexible polyhedron)  in~$\R^3$ was constructed by Connelly~\cite{Con77} in 1977.  Soon after that a simpler $9$-vertex flexor was found by Steffen~\cite{Ste78}. Steffen's flexor and its unfolding are shown in Fig.~\ref{fig_Steffen}. Up to now, it is the simplest known flexor. As was mentioned above, Bricard's results imply that there are no flexors with $6$ vertices. Maximov~\cite{Max06} proved that  there are no flexors with $7$ vertices, and made some progress towards the non-existence of flexors with $8$ vertices. Nevertheless, the following problem is still open.

\begin{figure}
\begin{center}
\includegraphics[scale=0.2]{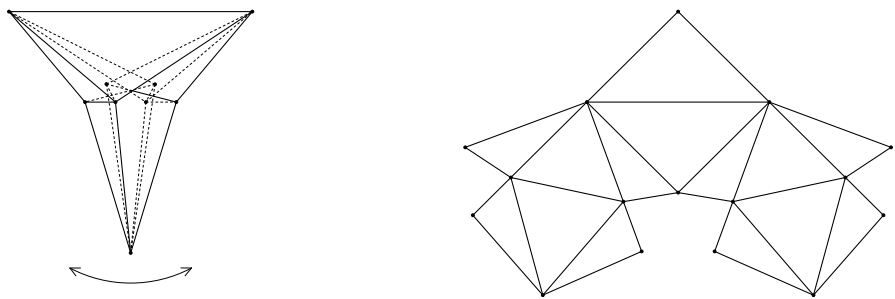}
\end{center}
\caption{Steffen's flexor and its unfolding}\label{fig_Steffen}
\end{figure}

\begin{problem}
Do there exist flexors in~$\R^3$ with $8$ vertices?
\end{problem}

Connelly's flexor and Steffen's flexor have topological type of sphere. Certainly, flexors of higher genera can be obtained by attaching rigid polyhedral handles  to these flexors. Such trivial examples are not interesting. An interesting problem was to construct flexors (or at least self-intersecting flexible polyhedra) of higher genera that are indecomposable in the sense that they cannot be decomposed into polyhedral surfaces of smaller genera with boundaries  such that these boundaries remain rigid during the flexions of the initial polyhedron.  The first example of an indecomposable self-intersecting flexible polyhedron of genus~$1$ was constructed by Alexandrov~\cite{Ale95}. Recently, Shtogrin~\cite{Sht15} constructed  flexors~$P^{(g)}_t$ of all genera~$g$ such that all but the one dihedral angle of~$P^{(g)}_t$ change non-trivially as $t$ runs over any interval. This property provides that the flexors~$P^{(g)}_t$ are indecomposable. 

There is a very strange phenomenon of the existence of a  constant dihedral angle in a flexor. Namely, each of the known examples of flexors including Connelly's flexor, Steffen's flexor, and Shtogrin's flexors of all genera contains at least one edge such that the dihedral angle at this edge is constant during the flexion. On the other hand there are self-intersecting flexible polyhedra, for example, Bricard's octahedra with all dihedral angles changing non-trivially during the flexion. The connection between embeddability and the existence of a constant dihedral angle is mysterious, and there are no reasonable arguments in favor of it. The following natural problem was posed by Sabitov, and still remains unsolved.

\begin{problem}
Does there exist a flexor~$P_t$, $t\in (\alpha,\beta)$, in~$\R^3$ such that all dihedral angles of~$P_t$ change non-trivially as $t$ runs over any subinterval of~$(\alpha,\beta)$?
\end{problem}

\section{Configuration spaces of flexible polyhedra}

Let us start with a rigorous definition of a simplicial polyhedron in~$\R^n$. A finite simplicial complex~$K$ is called a \textit{$k$-dimensional pseudo-manifold} if  every simplex of~$K$ is contained in a $k$-dimensional simplex,
 every $(k-1)$-dimensional simplex of~$K$ is contained in exactly two $k$-dimensional simplices,
and $K$ is \textit{strongly connected}. The latter means that any two $k$-dimensional simplices of~$K$ can be connected by a sequence of $k$-dimensional simplices such that every two consecutive simplices  have a common $(k-1)$-dimensional face. 
A pseudo-manifold~$K$ is  \textit{oriented} if its simplices of maximal dimension are endowed with compatible orientations. 

\begin{definition}\label{def_polyhedron}
Let $K$ be an oriented $(n-1)$-dimensional pseudo-manifold. A \textit{simplicial polyhedron} in~$\R^n$ of combinatorial type~$K$ is a mapping $P\colon K\to\R^n$ whose restriction to every simplex of~$K$ is affine linear. A polyhedron~$P$ is called \textit{non-degenerate} if the restriction of~$P$ to every simplex of~$K$ is an embedding,
and there is no decomposition of~$K$ into the union of  subcomplexes~$K_1,\ldots,K_s$ such that $\dim K_i=n-1$ for all~$i$, and $P(K_i\cap K_j)$ is contained  in an $(n-2)$-dimensional plane for any $i\ne j$.
 A \textit{flexible simplicial polyhedron} of combinatorial type~$K$ is a continuous family of polyhedra $P_t\colon K\to\R^n$, where $t$ runs over some interval, such that all edge lengths of~$P_t$ are constant as $t$ varies but the deformation~$P_t$ is not induced by an isometry of~$\R^n$. A flexible polyhedron is said to be \textit{non-degenerate} if $P_t$ is a non-degenerate polyhedron for all but finitely many values of~$t$. 
\end{definition}

Let us fix an oriented $(n-1)$-dimensional pseudo-manifold~$K$, and a set~$\bell$ of positive numbers~$\ell_{uv}=\ell_{vu}$ indexed by edges~$[uv]$ of~$K$. (If two vertices~$u$ and~$v$ are not joined by an edge, we do not fix any number~$\ell_{uv}$.) We would like to consider the space of all polyhedra $P\colon K\to\R^n$ with the prescribed set of edge lengths~$\bell$ up to orientation preserving isometries of~$\R^n$.

We shall always assume that the fixed set of edge lengths~$\bell$ satisfy the following condition: For any simplex~$[v_1\ldots v_k]$ of~$K$ there exists a non-degenerate simplex in~$\R^n$ with the prescribed edge lengths~$\ell_{v_iv_j}$. If this condition were not satisfied, then there would not exist  non-degenerate polyhedra~$P$ with the prescribed set of edge lengths.
Let us fix an $(n-1)$-dimensional simplex $[w_1\ldots w_n]$ of~$K$ and a non-degenerate simplex~$[\ba_{1}\cdots\ba_{n}]$  in~$\R^n$ such that the length of the edge~$[\ba_{i}\ba_{j}]$ is equal to~$\ell_{w_iw_j}$ for any $i\ne j$.
For each polyhedron $P\colon K\to\R^n$ with the set of edge lengths~$\bell$, there is a unique orientation preserving isometry of~$\R^n$ that takes~$P$ to~$P'$ such that $P'(w_i)=\ba_{i}$, $i=1,\ldots,n$. Hence imposing conditions $P(w_i)=\ba_{i}$ is equivalent to taking the quotient of the space of all polyhedra~$P$ with the set of edge lengths~$\bell$ by the group of orientation preserving isometries of~$\R^n$. 

We denote by~$\Sigma=\Sigma_{\R^n}(K,\bell)$ the space of all polyhedra $P\colon K\to\R^n$ that have the prescribed edge lengths~$\bell$ and satisfy $P(w_i)=\ba_{i}$, $i=1,\ldots,n$. The space~$\Sigma$ will be called the \textit{configuration space} of polyhedra with the prescribed combinatorial type~$K$ and the prescribed set of edge lengths~$\bell$. Let us show that~$\Sigma$ is a (possibly reducible) real affine algebraic  variety. Suppose that $K$ has $m$ vertices and~$r$ edges. For each vertex~$v$, we put $P(v)=\bx_v=(x_{v,1},\ldots,x_{v,n})$. The points~$\bx_{w_i}=\ba_i$ are fixed. If $v$ does not coincide with any of $w_1,\ldots,w_n$, then the coordinates~$x_{v,s}$ of~$\bx_v$ can be considered as independent variables. Consider the real affine space~$\R^{(m-n)n}$ with the coordinates~$x_{v,s}$, where $v$ runs over all vertices of~$K$ different from $w_1,\ldots,w_n$, and $s=1,\ldots,n$. Obviously, a polyhedron $P\colon K\to\R^n$ is uniquely determined by $m$ points~$\bx_v$. Hence a polyhedron  $P\colon K\to\R^n$ satisfying  $P(w_i)=\ba_{i}$ can be identified with the corresponding point in~$\R^{(m-n)n}$, which will be also denoted by~$P$. Thus the configuration space~$\Sigma$ is the affine variety in~$\R^{(m-n)n}$ given by $r-n(n-1)/2$ quadratic equations
\begin{equation}\label{eq_system}
|\bx_u-\bx_v|^2=\ell_{uv}^2,
\end{equation}
where $[uv]$ runs over all edges of~$K$ except for the edges~$[w_iw_j]$. (Notice that there are no equations for the diagonals of the polyhedron.)
It is easy to see that different choices of simplices $[w_1\ldots w_n]$ and $[\ba_1\ldots\ba_n]$ lead to isomorphic affine varieties~$\Sigma$.

If $n=3$ and $K$ is a sphere, then Euler's formula yields that $r=3m-6$. Hence~\eqref{eq_system} is a system of $3m-9$ equations in $3m-9$ variables. It can be shown that for a generic~$\bell$ these equations are algebraically independent, and there is a non-empty open subset $\mathcal{U}\subset\R^{3m-9}$ such that the system of equations~\eqref{eq_system} is compatible for any~$\bell\in \mathcal{U}$.
If $K$ is an oriented surface of genus~$g>0$, then~\eqref{eq_system} is a system of $3m-9+6g$ equations in $3m-9$ variables. Hence this system of equations is overdetermined and is incompatible for a generic~$\bell$.
If $n>3$, then the system of equations~\eqref{eq_system} is overdetermined in all interesting cases.

If the system of equations~\eqref{eq_system} is compatible, it generally has finitely many isolated solutions. These isolated solutions are rigid polyhedra of combinatorial type~$K$. However, if for certain set of edge lengths~$\bell$ the affine variety~$\Sigma$ has an irreducible component~$\Xi$ of positive geometric dimension, then we obtain a flexible polyhedron.  Notice that~$\Sigma$ may contain irreducible  components consisting of degenerate polyhedra. Such irreducible  components will be called \textit{inessential} and will be neglected. 

For each irreducible component~$\Xi$ of~$\Sigma$,  consider its Zariski closure in~$\C^{(m-n)n}$. (Here we regard~$\Xi$  as a subset of~$\R^{(m-n)n}\subset\C^{(m-n)n}$ forgetting about the ideal of polynomial equations by which it was initially given.) This Zariski closure will be denoted by~$\Xi_{\C}$ and will be called the \textit{complexification} of~$\Xi$. It is easy to see that~$\Xi_{\C}$ is an irreducible complex affine variety and $\dim_{\C}\Xi_{\C}=\dim_{\R}\Xi$.

Most of known flexible polyhedra admit one-parametric flexions only. This corresponds to the case $\dim\Xi=1$. Then $\Xi_{\C}$ is a complex curve. Surprisingly, for all known examples of flexible polyhedra this curve is either rational or elliptic (cf. Section~\ref{section_high}). Hence, the following problem seems to be interesting.

\begin{problem}
Does there exist one-parametric flexible polyhedra for which the complexification of the configuration space is a complex curve of genus greater than~$1$? More precisely, does there exist a pair $(K,\ell)$ such that the affine variety~$\Sigma_{\R^n}(K,\bell)$ has an essential one-dimensional irreducible component~$\Xi$ whose complexification is a complex curve of genus greater than~$1$? 
\end{problem} 

Alongside with flexible polyhedra in the Euclidean space~$\R^n$, one can study flexible polyhedra in non-Euclidean spaces of constant curvature, that is, in the Lobachevsky space~$\Lambda^n$ and in the round sphere~$S^n$. In the spherical case, to avoid uninteresting examples involving antipodal points one should usually restrict himself to considering only polyhedra contained in the open hemisphere~$S^n_+$. We shall always realise the sphere~$S^n$ as the unit sphere in the Euclidean space~$\R^{n+1}$ with centre at the origin, and realise the Lobachevsky space~$\Lambda^n$ as the sheet of the hyperboloid $\langle\bx,\bx\rangle=1$, $x_0>0$ in the pseudo-Euclidean space $\R^{1,n}$ with coordinates $x_0,\ldots,x_n$ and the scalar product
\begin{equation}\label{eq_pseudo_scalar}
\langle\bx,\by\rangle=x_0y_0-x_1y_1-\cdots-x_ny_n.
\end{equation}

Definition~\ref{def_polyhedron} can be literally repeated in the non-Euclidean case with the only exception: We cannot use affine linear mappings to map an affine simplex to either~$S^n$ or~$\Lambda^n$. Instead, we should use pseudo-linear mappings. A mapping~$f$ of an affine simplex  $[v_0\ldots v_k]$ to either $S^n$ or~$\Lambda^n$ is called \textit{pseudo-linear} if 
$$
f(\beta_0v_0+\cdots+\beta_kv_k)=\frac{\beta_0v_0+\cdots+\beta_kv_k}{\sqrt{\langle\beta_0v_0+\cdots+\beta_kv_k,\beta_0v_0+\cdots+\beta_kv_k\rangle}}
$$
for all $\beta_0,\ldots,\beta_k\ge 0$ such that $\beta_0+\cdots+\beta_k=1$, where $\langle\cdot,\cdot\rangle$ is the Euclidean scalar product in~$\R^{n+1}$ in the case of~$S^n$ and the pseudo-Euclidean scalar product~\eqref{eq_pseudo_scalar} in~$\R^{1,n}$ in the case of~$\Lambda^n$.

Let $X^n$ be either $S^n$ or~$\Lambda^n$. 
We shall conveniently put $c(t)=\cos t$ when $X^n=S^n$, and $c(t)=\cosh t$ when $X^n=\Lambda^n$.
Recall that the distance between points in~$X^n$ satisfy
$$
c(\dist_{X^n}(\bx,\by))=\langle\bx,\by\rangle.
$$

A definition of the configuration space~$\Sigma=\Sigma_{X^n}(K,\bell)$ of polyhedra $P$ in~$X^n$ of the prescribed combinatorial type~$K$ and the prescribed set of edge lengths~$\bell$ is similar to the Euclidean case. Again, we choose a simplex $[w_1\ldots w_n]$ of~$K$ and fix the points $P(w_i)=\ba_i$. Then for each vertex $v$ of~$K$ different from $w_1,\ldots,w_n$, the point~$P(v)$ is a vector $\bx_v=(x_{v,0},\ldots,x_{v,n})$ in either~$\R^{n+1}$ or~$\R^{1,n}$. Consider the real affine space~$\R^{(m-n)(n+1)}$ with the  coordinates~$x_{v,s}$, where $v$ runs over all vertices of~$K$ different from $w_1,\ldots,w_n$, and $s=0,\ldots,n$.
Now, we have quadratic equations of two types. First, we should impose the condition that all points~$\bx_v$ belong to~$X^n$. Hence we obtain $m-n$ equations
\begin{equation}\label{eq_system1}
\langle\bx_v,\bx_v\rangle=1,
\end{equation}
where $v$ runs over all vertices of~$K$ different from $w_1,\ldots,w_n$. Second, we should require that the polyhedron has the prescribed edge lengths. Therefore we obtain $r-n(n-1)/2$ equations
\begin{equation}\label{eq_system2}
\langle\bx_u,\bx_v\rangle=c(\ell_{uv}),
\end{equation}
where $[uv]$ runs over all edges of~$K$ except for the edges~$[w_iw_j]$.
Thus the configuration space~$\Sigma\subset\R^{(m-n)(n+1)}$ is the real affine variety given by the $r+m-n(n+1)/2$ equations~\eqref{eq_system1} and~\eqref{eq_system2}.

\section{Examples of high-dimensional flexible polyhedra}\label{section_high}

Until recently even self-intersecting flexible polyhedra were known only in spaces of dimensions~$3$ and~$4$.  We have discussed flexible polyhedra in~$\R^3$ in  Section~\ref{section_three}. All these flexible polyhedra have analogs in~$S^3$ and~$\Lambda^3$. This was first noticed by Kuiper~\cite{Kui78}.
Examples of self-intersecting flexible polyhedra in~$\R^4$ were constructed by Walz (unpublished) and by Stachel~\cite{Sta00}. These flexible polyhedra have combinatorial type of four-dimensional cross-polytope.

The \textit{regular $n$-dimensional cross-polytope} is the regular polytope dual to the $n$-dimensional cube, i.\,e., the convex hull of $2n$ points $\pm \mathbf{e}_1,\ldots,\pm\mathbf{e}_n$, where $\mathbf{e}_1,\ldots,\mathbf{e}_n$ is the standard basis of~$\R^n$. (We always identify a vector with its endpoint.) Denote by~$K^{n-1}$ the boundary of this polytope. A polyhedron $P$  of combinatorial type~$K^{n-1}$ will be called  a \textit{cross-polytope}. A cross-polytope~$P$ is uniquely determined by its $2n$ vertices $\ba_i=P(\mathbf{e}_i)$ and $\bb_i=P(-\mathbf{e}_i)$, $i=1,\ldots,n$.

In~\cite{Gai13} the author constructed examples of flexible cross-polytopes in the Euclidean spaces~$\R^n$, the Lobachevsky spaces~$\Lambda^n$, and  the round spheres~$S^n$  of all dimensions, and obtained a complete classification of all flexible cross-polytopes. In dimensions $5$ and higher, they became the first examples of flexible polyhedra. Besides, for any flexible cross-polytope was written an explicit parametrization for its flexion in either rational or elliptic functions. Now, we discuss briefly some of these results. 
As in the previous section, we fix the vertices~$\ba_1,\ldots,\ba_n$, regard the coordinates of the vertices $\bb_1,\ldots,\bb_n$ as~$n^2$ independent variables, and consider the configuration space $\Sigma$ of all cross-polytopes $P$ of combinatorial type~$K^{n-1}$ with the prescribed set of edge lengths~$\bell$ and the prescribed vertices~$\ba_1,\ldots,\ba_n$.

\begin{theorem}[\cite{Gai13}]
Any non-degenerate flexible cross-polytope admits not more than a one-parametric flexion. In other words, any essential irreducible component~$\Xi$ of~$\Sigma$ is either a point or a curve. In the latter case, $\Xi_{\C}$ is either a rational  or an elliptic complex curve. For each of the spaces~$\R^n$, $S^n$, and~$\Lambda^n$ of every dimension~$n$, there exist non-degenerate flexible cross-polytopes with both rational and elliptic curves~$\Xi_{\C}$.
\end{theorem}

\begin{problem}
Does there exists a set of edge lengths~$\bell$ of~$K^{n-1}$ such that the variety~$\Sigma=\Sigma_{X^n}(K^{n-1},\bell)$, where $X^n$ is~$\R^n$ or~$S^n$ or~$\Lambda^n$, contains two different essential one-dimensional irreducible components?
\end{problem}

Here we shall not describe a complete classification of flexible cross-polytopes obtained in~\cite{Gai13} but we shall explain some ideas behind the construction of high-dimensional flexible cross-polytopes and we shall write explicitly  parametrizations for two examples of flexible cross-polytopes. For simplicity, we shall restrict ourselves to the Euclidean case.

Let us introduce some notation. Denote by $a_1,\ldots,a_n$ the lengths of the altitudes of the simplex $[\ba_1\ldots\ba_n]$ drawn from the vertices $\ba_1,\ldots,\ba_n$, respectively. Denote by $\bn_1,\ldots,\bn_n$ the interior unit normal vectors to the facets of the simplex $[\ba_1\ldots\ba_n]$ opposite to the vertices $\ba_1,\ldots,\ba_n$, respectively. Put $g_{ij}=\langle\bn_i,\bn_j\rangle$, in particular, $g_{ii}=1$. Choose one of the two unit normal vectors to the hyperplane in~$\R^n$ spanned by the simplex $[\ba_1\ldots\ba_n]$, and denote this vector by~$\bm$.

In his original paper~\cite{Bri97} Bricard deduced the following equation describing the  flexions of a tetrahedral angle. Let $\bs\ba\bb\bc\bd$ be a tetrahedral angle with vertex~$\bs$ and consecutive edges $\bs\ba$, $\bs\bb$, $\bs\bc$, and~$\bs\bd$. Assume that the flat angles $\ba\bs\bb$, $\bb\bs\bc$, $\bc\bs\bd$, and~$\bd\bs\ba$  are rigid plates, and there are hinges at the edges~$\bs\ba$, $\bs\bb$, $\bs\bc$, and~$\bs\bd$. Denote by~$\phi$ and~$\psi$ the dihedral angles of the tetrahedral angle~$\bs\ba\bb\bc\bd$ at the edges~$\bs\ba$ and~$\bs\bb$, respectively. Then the values $t=\tan(\phi/2)$ and~$t'=\tan(\psi/2)$ satisfy a biquadratic relation of the form
\begin{equation}\label{eq_biquad}
At^2{t'}^2+Bt^2+2Ctt'+D{t'}^2+E=0,
\end{equation}
where the coefficients $A$, $B$, $C$, $D$, and~$E$ can be written explicitly from the values of the flat angles $\ba\bs\bb$, $\bb\bs\bc$, $\bc\bs\bd$, and~$\bd\bs\ba$. Then Bricard wrote the three equations of form~\eqref{eq_biquad} for the tetrahedral angles at the vertices~$\ba_1$, $\ba_2$, and~$\ba_3$ of an octahedron $P\colon K^2\to\R^3$. Thus he obtained three biquadratic equations in the three variables $t_i=\tan(\phi_i/2)$, $i=1,2,3$, where $\phi_1$, $\phi_2$, and~$\phi_3$ are the dihedral angles of the octahedron at the edges~$[\ba_2\ba_3]$, $[\ba_3\ba_1]$, and~$[\ba_1\ba_2]$, respectively. An octahedron~$P$ is flexible if and only if the obtained system of equations has a one-parametric family of solutions. Further, Bricard used this fact to prove that any flexible octahedron has certain special geometric properties. Namely, either certain edges or certain flat angles of the octahedron should be pairwise equal to each other. Finally, these geometric properties  were used to obtain a complete classification of flexible octahedra. 

In higher dimensions this geometric approach does not work, since flexible cross-polytopes typically have neither symmetries nor equal edges or angles.  However, the system of equations of the form~\eqref{eq_biquad} also can be written and plays the key role in our construction. The difference of our approach from Bricard's approach is that instead of trying to deduce geometric consequences of these equations, we study the compatibility conditions for this system of equations from algebraic viewpoint.

First, let us show how to obtain a system of equations of the form~\eqref{eq_biquad} in an arbitrary dimension. For an $n$-dimensional cross-polytope $P$, we denote by~$\phi_i$ the dihedral angle of it at the $(n-2)$-dimensional face $F_i=[\ba_1\ldots\hat\ba_i\ldots\ba_n]$, where hat denotes the omission of the vertex. We put $t_i=\tan(\phi_i/2)$. Consider the $(n-3)$-dimensional face $F_{ij}=[\ba_1\ldots\hat\ba_i\ldots\hat\ba_j\ldots\ba_n]$, and intersect it by a three-dimensional plane~$L$ orthogonal to it. Then the intersections of~$L$ with the four $(n-1)$-dimensional faces of~$P$ containing~$F_{ij}$ are flat angles that form a tetrahedral angle. The dihedral angles of this tetrahedral angle at the two consecutive edges~$F_i\cap L$ and~$F_j\cap L$  are equal to~$\phi_i$ and~$\phi_j$, respectively. Thus we obtain an equation 
\begin{equation}\label{eq_biquad2}
A_{ij}t_i^2t_j^2+B_{ij}t_i^2+2C_{ij}t_it_j+D_{ij}t_j^2+E_{ij}=0,
\end{equation}
where the coefficients $A_{ij}$, $B_{ij}$, $C_{ij}$, $D_{ij}$, and~$E_{ij}$ can be written explicitly from the set of edge lengths~$\bell$. The obtained system of $n(n-1)/2$ equations in $n$ variables is overdetermined when~$n>3$. 
The problem of classifying flexible cross-polytopes comes to the problem of finding of all~$\bell$ such that the system of equations~\eqref{eq_biquad2} has a one-parametric family of solutions, which seems to be rather hard. Nevertheless, this problem can be solved in the following way. First, we find all systems of functions $t_1(u),\ldots,t_n(u)$ that satisfy the system of non-trivial biquadratic equations of the form~\eqref{eq_biquad2} with \textit{some} coefficients $A_{ij}$, $B_{ij}$, $C_{ij}$, $D_{ij}$, and~$E_{ij}$ not necessarily corresponding to any set of edge lengths~$\bell$. Second, we solve the problem of reconstructing of the geometry of~$P$ (equivalently, of the set of edge lengths~$\bell$) from the coefficients $A_{ij}$, $B_{ij}$, $C_{ij}$, $D_{ij}$, and~$E_{ij}$. This program was realized in~\cite{Gai13}. Let us illustrate it with two instructive examples.

First, consider the functions $t_i(u)=\lambda_iu$, $i=1,\ldots,n$, where $\lambda_i$ are nonzero real numbers such that $\lambda_i\ne \pm\lambda_j$ whenever $i\ne j$. These functions satisfy infinitely many systems of equations of the form~\eqref{eq_biquad2}. Indeed, we obtain that, for any~$i\ne j$, $A_{ij}=E_{ij}=0$, and the three coefficients $B_{ij}$, $C_{ij}$, and~$D_{ij}$ satisfy the equation $B_{ij}\lambda_i^2+2C_{ij}\lambda_i\lambda_j+D_{ij}\lambda_j^2=0$, which has infinitely many non-trivial solutions. These solutions lead to the following family of non-degenerate flexible cross-polytopes, see~\cite[Sect.~5]{Gai13}. Choose any non-degenerate $(n-1)$-dimensional simplex $[\ba_1\ldots\ba_n]$, which will remain fixed during the flexion, and choose any nonzero real numbers $\lambda_1,\ldots,\lambda_n$ such that $\lambda_i\ne \pm\lambda_j$ whenever $i\ne j$. Then the motion of the vertices $\bb_i$ during the flexion  is parametrized by
\begin{equation}\label{eq_param1}
\begin{split}
\bb_i(u)&=\left(\frac{1}{a_i}+2\lambda_i\sum_{j\ne i}\frac{\lambda_ig_{ij}-\lambda_j}{a_j(\lambda_i^2-\lambda_j^2)}\right)^{-1}\\
&\times \left(\frac{\ba_i}{a_i}+2\lambda_i\sum_{j\ne i}\frac{(\lambda_ig_{ij}-\lambda_j)\ba_j}{a_j(\lambda_i^2-\lambda_j^2)} 
+\frac{2\lambda_iu(\bm-\lambda_iu\,\bn_i)}{\lambda_i^2u^2+1}\right)
\end{split}
\end{equation}

Notice that, once this formula is written, the constancy of the edge lengths of the cross-polytope can be checked by an easy immediate computation. It follows directly from the construction that the tangents of the halves of dihedral angles of this cross-polytope are proportional to each other during the flexion\footnote{We were not precise enough in our consideration of dihedral angles. In fact, in some cases we should take the interior dihedral angles, and in other cases we should take the exterior dihedral angles, which are obtained from the interior dihedral angles by subtracting them from~$\pi$. Hence the correct statement is as follows:  The tangents of the halves of dihedral angles of the cross-polytope  are either directly or inversely proportional to each other during the flexion.}. If $n=3$, then the obtained family of flexible cross-polytopes turns into Bricard's flexible octahedra of the third type. For them, the fact that the tangents of the halves of dihedral angles are either directly or inversely proportional to each other during the flexion was known to Bricard~\cite{Bri97}. Surprisingly, the simplest from the algebraic viewpoint family of flexible cross-polytope turns into the most complicated from the geometric viewpoint type of flexible octahedra.

More complicated one-parametric families of solutions of the systems of equations~\eqref{eq_biquad2} can be written in Jacobi's elliptic functions. The first who noticed that flexions of spherical quadrilaterals, which are equivalent to flexions of tetrahedral angles, can be parametrized in elliptic functions was Darboux~\cite{Dar79}. Later Connelly~\cite{Con74} used the Weierstrass $\wp$-function to parametrize flexible polyhedra in three-space. However, their methods for introducing the elliptic parametrization were not based on equations of the form~\eqref{eq_biquad}, hence, were not appropriate for generalising to higher dimensions. Equation~\eqref{eq_biquad} is closely related to addition laws for Jacobi's elliptic functions. Indeed, if we fix any elliptic modulus~$k$, $0<k<1$, and take, say, $t(u)=\dn u$, $t'(u)=\dn(u-\sigma)$, then these functions will satisfy equation~\eqref{eq_biquad} with coefficients $A=\sn^2\sigma$, $B=D=\cn^2\sigma$, $C=\dn\sigma$, and $E=(1-k^2)\sn^2\sigma$. The first who noticed that this fact can be used in theory of flexible polyhedra was Izmestiev~\cite{Izm14},~\cite{Izm15}\footnote{Though papers~\cite{Izm14},~\cite{Izm15} were put on the arXiv later than~\cite{Gai13}, certain preliminary versions of them circulated as preprints before the paper~\cite{Gai13} was written, and the author borrowed from them the idea of using the elliptic parametrization for the solutions of equations of the form~\eqref{eq_biquad}.}. He used  elliptic solutions of equations of the form~\eqref{eq_biquad} to study the flexions of the so-called Kokotsakis polyhedra with quadrangular base, where a Kokotsakis polyhedron is  a polyhedral surface with boundary in three-space that is combinatorially equivalent to a neighborhood of a quadrilateral in a quad surface. 

Returning to the system of equations~\eqref{eq_biquad2}, we can write many different solutions of it in elliptic functions; all of them are classified in~\cite[Sects.~6,\,7]{Gai13}. Here we present only one example. Choose an elliptic modulus~$k\in(0,1)$, real phases~$\sigma_1,\ldots,\sigma_n$ that are pairwise different modulo $K\Z$, where $K$ is the real quarter-period corresponding to the modulus~$k$, and nonzero coefficients~$\lambda_1,\ldots,\lambda_n$. Put $t_i(u)=\lambda_i\dn(u-\sigma_i)$, $i=1,\ldots,n$. These functions satisfy a system of equations of the form~\eqref{eq_biquad2} but, unlike the previous case, the coefficients of these equations are determined uniquely up to proportionality. This implies that we cannot choose the simplex $[\ba_1\ldots\ba_n]$ arbitrarily. Instead, we should choose  this simplex so that the elements of the Gram matrix of $\bn_1,\ldots,\bn_n$ are given by $g_{ii}=1$ and 
\begin{equation}\label{eq_G}
g_{ij}=\frac{(\lambda_i^2+\lambda_j^2)\cn^2(\sigma_i-\sigma_j)-(1+(1-k^2)\lambda_i^2\lambda_j^2)\sn^2(\sigma_i-\sigma_j)}
{2\lambda_i\lambda_j\dn(\sigma_i-\sigma_j)}\,,\qquad i\ne j.
\end{equation}
Here we face the following difficulty. Not any symmetric real matrix with units on the diagonal can be realised as the Gram matrix of unit vectors orthogonal to the facets of a simplex.  This matrix must be degenerate positive semidefinite, and must have nonzero proper principal minors. However, it can be shown that the parameters~$k$, $\sigma_1,\ldots,\sigma_n$, $\lambda_1,\ldots,\lambda_n$ can be chosen so that the matrix $G=(g_{ij})$ given by~\eqref{eq_G} will satisfy these conditions. We again denote by~$a_i$ the lengths of the altitudes of the simplex $[\ba_1\ldots\ba_n]$. Certainly, now they can be written explicitly (up to proportionality) from the Gram matrix elements~$g_{ij}$, hence, from~$k$, $\sigma_1,\ldots,\sigma_n$, $\lambda_1,\ldots,\lambda_n$  but the resulting expressions will be too cumbersome. The parametrization of the flexible cross-polytope is now given by
\begin{equation}\label{eq_param2}
\begin{split}
\bb_i(u)&=\left(\frac{1}{a_i}+\lambda_i\sum_{j\ne i}\frac{\cn^2(\sigma_i-\sigma_j)-(1-k^2)\lambda_j^2\sn^2(\sigma_i-\sigma_j)}{a_j\lambda_j\dn(\sigma_i-\sigma_j)}\right)^{-1}\\
{}&\times \left(\frac{\ba_i}{a_i}+\lambda_i\sum_{j\ne i}\frac{(\cn^2(\sigma_i-\sigma_j)-(1-k^2)\lambda_j^2\sn^2(\sigma_i-\sigma_j))\ba_j}{a_j\lambda_j\dn(\sigma_i-\sigma_j)}\right.\\
&\hspace{11.4mm}{}+\left.\frac{2\lambda_i\dn (u-\sigma_i)\bm-2\lambda_i^2\dn^2 (u-\sigma_i)\,\bn_i}{\lambda_i^2\dn^2(u-\sigma_i)+1}
\vphantom{\frac{\ba_i}{a_i}+\lambda_i\sum_{j\ne i}\frac{(\cn^2(\sigma_i-\sigma_j)-(1-k^2)\lambda_j^2\sn^2(\sigma_i-\sigma_j))\ba_j}{a_j\lambda_j\dn(\sigma_i-\sigma_j)}}
\right)
\end{split}
\end{equation}
For $n=3$, this flexible cross-polytope is Bricard's octahedron of the first type.

Though a complete classification of all flexible cross-polytopes in~$\R^n$, $S^n$, and~$\Lambda^n$ was obtained in~\cite{Gai13}, it is very hard to find out from a parametrization like~\eqref{eq_param1} or~\eqref{eq_param2} whether the flexible cross-polytope given by it is embedded or self-intersecting. However, the following conjecture seems to be plausible.

\begin{conjecture}
All flexible cross-polytopes in~$\R^n$ and in~$\Lambda^n$, where $n\ge 3$, are self-intersecting.
\end{conjecture}

In the spheres and even in open hemispheres a similar assertion is false.

\begin{theorem}[\cite{Gai15a}]
For each $n\ge 3$, there exist embedded flexible cross-polytopes in the open hemisphere~$S^n_+$.
\end{theorem}

\begin{problem}
Do there exist embedded flexible polyhedra in~$\R^n$ or~$\Lambda^n$ for $n\ge 4$?
\end{problem}

\section{The bellows conjecture}\label{section_BC}

Soon after the first examples of flexors had been found~\cite{Con77},~\cite{Ste78}, it was discovered that their volumes remain constant during the flexion, and the following conjecture was proposed, see~\cite{Kui78},~\cite{Con78}.

\begin{conjecture}[The bellows conjecture]
The volume of any flexor in~$\R^3$ is constant during the flexion.
\end{conjecture} 

This conjecture can be generalised to the case of a not necessarily embedded flexible polyhedra. To do this, one needs to introduce the concept of a generalised oriented volume of an arbitrary polyhedron $P:K\to\R^n$. If $P$ is an embedding, then under the volume of~$P$ we mean the volume of the region bounded by the polyhedral surface~$P(K)$. It is natural to say that the surface~$P(K)$ is \textit{positively oriented} if the pullback by~$P$ of its orientation given by the exterior normal at a smooth point coincides with the chosen orientation of~$K$, and is \textit{negatively oriented} if these two orientations are opposite to each other. If $P(K)$ is positively oriented, then we define the  \textit{characteristic function}~$\lambda_P(\bx)$ of~$P$ to be a piecewise constant function on~$\R^n$ that is equal to~$1$ inside the polyhedral surface~$P(K)$, is equal to~$0$ outside the polyhedral surface~$P(K)$, and is undefined on~$P(K)$. Similarly, if $P(K)$ is negatively oriented, then, by definition, the \textit{characteristic function}~$\lambda_P(\bx)$, is equal to~$-1$ inside~$P(K)$ and is equal to~$0$ outside~$P(K)$. Let us define the \textit{characteristic function}~$\lambda_P(\bx)$ of a not necessarily embedded polyhedron $P\colon K\to\R^n$ in the following way. For each point~$\bx\notin P(K)$, we take a generic curve~$\gamma$ going from~$\bx$ to infinity, and denote by~$\lambda_P(\bx)$ the algebraic intersection number of the curve~$\gamma$ and the $(n-1)$-dimensional cycle~$P(K)$. It is easy to see that this intersection index is independent of the choice of~$\gamma$. Then~$\lambda_P(\bx)$ is a piecewise constant function on~$\R^n$ undefined on~$P(K)$. By definition, the \textit{generalised oriented volume} of a polyhedron $P\colon K\to\R^n$ is given by
\begin{equation*}
V_K(P)=\int_{\R^n}\lambda_P(\bx)\,dV,
\end{equation*}
where $dV$ is the standard volume element in~$\R^n$. For an embedded polyhedron, the generalised oriented volume is exactly  the oriented volume of the region bounded by~$P(K)$.  A more general version of the bellows conjecture is as follows.

\begin{conjecture}[The bellows conjecture]
The generalised oriented volume of any flexible polyhedron in~$\R^n$, $n\ge 3$, is constant during the flexion.
\end{conjecture}

One of the most important breakthroughs in theory of flexible polyhedra was the proof of the bellows conjecture for flexible polyhedra in~$\R^3$ by Sabitov~\cite{Sab96}, see also~\cite{Sab98a},~\cite{Sab98b}. Another proof was obtained by Connelly, Sabitov, and Walz~\cite{CSW97}. The proof of the bellows conjecture was based on a wonderful discovery by Sabitov of the fact that the generalised oriented volume of any (not necessarily flexible) simplicial polyhedron satisfies a monic polynomial equation with coefficients determined solely by the combinatorial structure and the edge lengths of the polyhedron. This result contrasts to the two-dimensional case, since the only polygon for which the (generalised) oriented area satisfies such monic polynomial equation is a triangle whose area is given by Heron's formula. A good survey of the works on the three-dimensional bellows conjecture as well as of some other results and problems on flexible polyhedra can be found in~\cite{Sab11}.

Sabitov's approach to the proof of the existence of a monic polynomial equation for the volume cannot be generalised to higher dimensions, see below. Nevertheless, the author suggested a new approach that yielded the same result for the Euclidean spaces of all dimensions $n\ge 4$. The following theorem is due to Sabitov~\cite{Sab96} for $n=3$ (see also~\cite{Sab98a},~\cite{Sab98b},~\cite{CSW97}) and to the author~\cite{Gai11},~\cite{Gai12} for $n\ge 4$. 

\begin{theorem}\label{theorem_relation}
Let $K$ be an oriented $(n-1)$-dimensional pseudo-manifold,  $n\ge 3$. For a simplicial polyhedron $P\colon K\to\R^n$, we denote by~$\mathbf{q}$ the set of the squares of the edge lengths of~$P$, and by~$V$ the generalised oriented volume of~$P$. Then there exists a monic with respect to~$V$ polynomial relation
\begin{equation*}
V^{2N}+a_1(\mathbf{q})V^{2N-2}+a_2(\mathbf{q})V^{2N-4}+\cdots+a_N(\mathbf{q})=0
\end{equation*} 
that holds for all polyhedra $P\colon K\to\R^n$ of combinatorial type~$K$. Here $a_j(\mathbf{q})$ are polynomials with rational coefficients, and the numbers~$N$ and the polynomials~$a_j(\mathbf{q})$ are determined solely by the pseudo-manifold~$K$.
\end{theorem}

A monic with respect to~$V$ polynomial~$Q(V,\mathbf{q})$ such that $Q(V,\mathbf{q})=0$ for all polyhedra $P\colon K\to\R^n$ is called a \textit{Sabitov polynomial} for polyhedra of combinatorial type~$K$.

\begin{corollary}\label{cor_BC}
The generalised oriented volume of any flexible polyhedron in~$\R^n$, $n\ge 3$, is constant during the flexion. 
\end{corollary}

\begin{proof}
Any closed polyhedral surface in~$\R^n$ has a simplicial subdivision. Passing to this subdivision, we introduce new hinges. Hence, all flexions that have existed before, do still exist, and some new flexions may appear. Therefore the assertion of Corollary~\ref{cor_BC} for arbitrary flexible polyhedra will follow immediately from the assertion of Corollary~\ref{cor_BC} for simplicial flexible polyhedra.

Since a nonzero polynomial has finitely many roots,  Theorem~\ref{theorem_relation} implies that the generalised oriented volume of a simplicial polyhedron of the prescribed combinatorial type~$K$ and the prescribed set of edge lengths~$\bell$ can take only finitely many values. On the other hand, the generalised oriented volume of a flexible polyhedron changes continuously. Hence it is constant. 
\end{proof}

\begin{remark}
In fact, the result obtained by the author in~\cite{Gai12} is stronger than Theorem~\ref{theorem_relation}. Namely, we can replace the requirement that the polyhedron is simplicial with a weaker requirement that all two-dimensional faces of the polyhedron are triangles. (In dimension~$3$ these two conditions are equivalent.) This implies that the volume of a polyhedron remains constant not only during flexions but during all deformations such that the combinatorial type does not change and all two-dimensional faces remain congruent to themselves.
\end{remark}

Now, let us discuss some ideas behind the proof of Theorem~\ref{theorem_relation}. For any five points in~$\R^3$, the squares of the pairwise distances between them satisfy  a polynomial relation, which is called the Cayley--Menger relation and is equivalent to the degeneracy of the Gram matrix of the vectors from one of the points to the other four points. For a polyhedron in~$\R^3$, the Cayley--Menger relations for $5$-tuples of its vertices yield a system of polynomial relations among the squares of the lengths of edges and diagonals. Sabitov's original prove of Theorem~\ref{theorem_relation} for $n=3$ was based on a rather complicated technique for elimination the squares of the lengths of diagonals by means of resultants. Later, Connelly, Sabitov, and Walz~\cite{CSW97} noticed that this technique can be replaced with the usage of theory of places of fields, which makes the proof more involved but less cumbersome.

Recall that a \textit{place} of a field~$E$ is a mapping $\phi\colon E\to F\cup\{\infty\}$ to a field~$F$, with an extra element~$\infty$, such that $\phi(1)=1$, $\phi(a+b)=\phi(a)+\phi(b)$ and $\phi(ab)=\phi(a)\phi(b)$ whenever the right-hand sides are defined. Here we assume that $c+\infty=\infty$ for all $c\in F$, and $c\cdot\infty=\infty$ for all $c\in F\cup\{\infty\}\setminus\{0\}$. The expressions $\infty+\infty$ and $\infty\cdot 0$ are undefined. Elements $c\in F$ are said to be \textit{finite\/}. 
\begin{lemma}[cf.~{\cite[p.~12]{Lan72}}]\label{lem_place}
Let $R$ be a ring with unity contained in a field~$E$, and let $a$ be an element of~$E$. Then $a$ is integral over~$R$ if and only if every place~$\phi$ of~$E$ that is finite on~$R$ is finite on~$a$. 
\end{lemma}
This lemma is applied in the following way. Take for~$E$ the field~$\Q(\{x_{v,s}\})$ of rational functions in the coordinates~$x_{v,s}$ of vertices of the polyhedron. Then the squares of edge lengths~$q_{uv}=\ell_{uv}^2$ are elements of~$E$. Take for~$R$ the $\Q$-algebra generated by all~$q_{uv}$ such that $[uv]$ is an edge of~$K$, and take for~$a$ the generalised oriented volume~$V$, which is also an element of~$E$. Then one needs to prove that every place~$\phi$ of~$E$ that is finite on~$R$ is finite on~$V$.

Though the algebraic tools used in Sabitov's original proof and in the proof due to Connelly, Sabitov, and Walz  are different, both proofs use the same induction on certain parameters of the polyhedron (genus, number of vertices). The key role in both proofs is played by the fact that a two-dimensional simplicial manifold can be simplified by certain local moves monotonically with respect to the number of vertices, and moreover, there is a rich selection of opportunities to do so. For pseudo-manifolds  (and even for manifolds) of dimensions~$3$ and higher,  no analog of this fact is true. Though the author~\cite{Gai11}  generalized both proofs to polyhedra in~$\R^4$ by using a more delicate induction, and an additional result of algebraic geometry, it became clear that this approach cannot be used in higher dimensions.
 
Our approach in~\cite{Gai12}, which allowed to prove Theorem~\ref{theorem_relation} for all~$n$, is different. Instead of proceeding by induction on some parameters of the polyhedron, we temporarily forget about the polyhedron and study the properties of places of the field~$\Q(\{x_{i,s}\})$ of rational functions in the coordinates of~$m$ points~$\bx_1,\ldots,\bx_m$ in~$\R^n$. The key lemma is as follows.

\begin{lemma}\label{lem_collapse}
Let $\phi\colon \Q(\{x_{i,s}\})\to F\cup\{\infty\}$ be a place. Let $\Gamma_{\phi}$ be the graph on the vertex set~$\{1,\ldots,m\}$ such that $[ij]$ is  an edge of~$\Gamma_{\phi}$ if and only if $\phi(q_{ij})\ne\infty$, where $q_{ij}=|\bx_i-\bx_j|^2$. Let $K_{\phi}$ be the clique complex of~$\Gamma_{\phi}$, i.\,e., the simplicial complex whose simplices are spanned by cliques of~$\Gamma_{\phi}$. Then~$K_{\phi}$ collapses on a subcomplex of dimension less than~$[n/2]$.
\end{lemma} 

Each simplicial complex~$K$ on~$m$ vertices can be naturally considered as a subcomplex of the $(m-1)$-dimensional simplex~$\Delta$ with the same vertices. If~$K$ is an oriented $(n-1)$-dimensional pseudo-manifold, then its fundamental class~$[K]$ becomes a boundary in the simplicial chain complex of~$\Delta$, since $\Delta$ is contractible. For  each polyhedron $P\colon K\to\R^n$, the mapping~$P$ can be extended to an affine linear mapping of~$\Delta$ to~$\R^n$, which will be also denoted by~$P$. For each $n$-dimensional simplicial chain $\xi$ in~$\Delta$ such that $\partial\xi=[K]$,  the image $P(\xi)$ can be naturally regarded as a \textit{generalised triangulation} of the interior of~$P(K)$. In particular, if $\xi=\sum_kc_k\Delta_k$, then $V_K(P)=\sum_kc_kV_{\mathrm{or}}(P(\Delta_k))$, where $V_{\mathrm{or}}$ denotes the oriented volume of an $n$-dimensional simplex in~$\R^n$. 

Let $\phi$ be a place of the field~$\Q(\{x_{v,s}\})$ of rational functions in the coordinates of vertices of the polyhedron of combinatorial type~$K$. If $\phi$ is finite on all squares of the edge lengths of the polyhedron, then $K\subseteq K_{\phi}$. Since $n\ge 3$, we have $[n/2]<n-1$. Hence Lemma~\ref{lem_collapse} implies that the $(n-1)$-dimensional homology group of~$K_{\phi}$ vanishes. Therefore a chain~$\xi$ satisfying $\partial\xi=[K]$ can be chosen so that its support is contained in~$K_{\phi}$. (The \textit{support} of a chain is the union of all simplices entering this chain with nonzero coefficients.) Then~$\phi$ is finite on~$q_{uv}$ for all edges~$[uv]$ of the support of~$\xi$.
(Notice that edges of the support of~$\xi$ may be diagonals of the initial polyhedron.) It follows easily that~$\phi$ is finite on the oriented volume of any $n$-dimensional simplex entering~$\xi$, hence, is finite on~$V=V_K(P)$, which completes the proof of Theorem~\ref{theorem_relation}.

\section{The bellows conjecture for non-Euclidean spaces}

The definition of a \textit{generalised oriented volume} of a polyhedron $P\colon K\to\Lambda^n$ is literally the same as for~$\R^n$. The spherical case is more difficult even for embedded polyhedra: We do not know which of the two connected  components of the space $S^n\setminus P(K)$ should be considered as the interior of the polyhedron. For arbitrary polyhedra, this phenomenon becomes apparent as follows. We cannot define the characteristic function~$\lambda_P(\bx)$, since there is no infinity in the sphere. This difficulty can be overcome in the following way. For a polyhedron $P\colon K\to S^n$, we can define its generalised oriented volume as an element of the group~$\R/\sigma_n\Z$, where $\sigma_n$ is the volume of~$S^n$. Indeed, for each point~$\by\in S^n\setminus P(K)$, we can define a piecewise constant function~$\lambda_{P,\by}(\bx)$ on~$S^n$, which will be called the \textit{characteristic function of~$P$ with respect to~$\by$}, by computing  the algebraic intersection number of a generic curve~$\gamma$ from~$\bx$ to~$\by$ with the $(n-1)$-dimensional cycle~$P(K)$. Then 
\begin{equation}\label{eq_V_Sn}
V_P(K)=\int_{S^n}\lambda_{P,\by}(\bx)\,dV_{S^n}(\bx)\pmod{\sigma_n\Z},
\end{equation}
where $dV_{S^n}(\bx)$ is the standard volume element in~$S^n$. The characteristic functions of~$P$ with respect to two  points~$\by_1$ and~$\by_2$ differ by an integral constant. Hence the corresponding integrals in the right-hand side of~\eqref{eq_V_Sn} differ by an element of~$\sigma_n\Z$. Thus the generalised oriented volume of~$P$ is well defined as an element of~$\R/\sigma_n\Z$.

The bellows conjecture is obviously not true for polyhedra in~$S^n$ if we allow them to contain two antipodal points. Indeed, consider a flexible polygon with non-constant area in the equatorial great sphere~$S^2\subset S^3$, and take the bipyramid (the suspension) over it with vertices at the poles of~$S^3$. We obtain a flexible polyhedron in~$S^3$ with non-constant volume. Iterating this construction, we obtain flexible polyhedra with non-constant volumes in the spheres~$S^n$ for all  $n\ge 3$. Hence the bellows conjecture for non-Euclidean spaces was usually formulated as follows.

\begin{conjecture}[The bellows conjecture for non-Euclidean spaces] 
The generalised oriented volume of any flexible polyhedron in either the Lobachevsky space~$\Lambda^n$ or the open hemisphere~$S^n_+$  is constant during the flexion, provided that $n\ge 3$. 
\end{conjecture}

In 1997 Alexandrov constructed an example of a flexible polyhedron in~$S^3_+$ with non-constant generalised oriented volume, which disproved the bellows conjecture in~$S^3_+$. So the general expectation was that the bellows conjecture is not true for any non-Euclidean space. This expectation was supported by the following result obtained by the author~\cite{Gai15a} using the classification of flexible cross-polytopes.

\begin{theorem}[\cite{Gai15a}]
For any $n\ge 3$, there exist embedded flexible cross-polytopes in~$S^n_+$ with non-constant volumes.
\end{theorem}

So the bellows conjecture for~$S^n_+$ is false. The more surprising is that the bellows conjecture is true at least for odd-dimensional Lobachevsky spaces. 

\begin{theorem}[\cite{Gai15b}]\label{theorem_BC_Lob}
The generalised oriented volume of any bounded flexible polyhedron in the odd-dimensional Lobachevsky space~$\Lambda^{2k+1}$, where $k\ge 1$, is constant during the flexion.
\end{theorem}

Unlike the Euclidean spaces, in the Lobachevsky spaces there are unbounded flexible polyhedra of finite volume that have some vertices on the absolute. 

\begin{problem}
Is the bellows conjecture  true for unbounded flexible polyhedra in odd-dimensional  Lobachevsky spaces?
\end{problem}

Also, it is still unknown if the bellows conjecture is true for even-dimensional Lobachevsky spaces.

Theorem~\ref{theorem_BC_Lob} makes plausible that certain weaker form of the bellows conjecture for spheres is still true despite of existing counterexamples to the initial version of the conjecture. The following result by the author~\cite{Gai16} shows that the bellows conjecture holds true for all sufficiently small polyhedra in all non-Euclidean spaces.

\begin{theorem}[\cite{Gai16}]\label{theorem_small}
Let $X^n$ be either~$S^n$ or~$\Lambda^n,$ $n\ge 3$. Let $P_t\colon K\to X^n$ be a simplicial flexible polyhedron with $m$ vertices such that all edges of~$P_t$ have lengths smaller than $2^{-m^2(n+4)}$. Then the generalized oriented volume of~$P_t$ is constant during the flexion.
\end{theorem}

\begin{problem}
Suppose that $n\ge 3$. Does there exist a constant $\varepsilon_{X^n}>0$ depending only on the space~$X^n$ such that the generalized oriented volumes of all flexible polyhedra in~$X^n$ of diameters less than~$\varepsilon_{X^n}$ are constant during the flexion?
\end{problem}

Notice that the author's classification of flexible cross-polytopes~\cite{Gai13} implies that in each space~$X^n$ there exist flexible cross-polytopes with arbitrarily small edge lengths, hence, the assertion of Theorem~\ref{theorem_small} is not empty. 

In non-Euclidean spaces there is no hope to obtain any reasonable analog of Theorem~\ref{theorem_relation} providing a way to compute the volume of a simplicial polyhedron from its edge lengths, since even in the simplest case of a tetrahedron in~$\Lambda^3$ or~$S^3$ all known formulae for the volume from the edge lengths are very complicated. The proofs of Theorems~\ref{theorem_BC_Lob} and~\ref{theorem_small} are based on the study of the analytic continuation of the volume function~$V_K$ defined on the configuration space~$\Sigma=\Sigma_{X^n}(K,\bell)$ to the complexification of~$\Sigma$. We  focus on the proof of Theorem~\ref{theorem_BC_Lob}.
More precisely, we take the canonical stratification of~$\Sigma$ built by Whitney~\cite{Whi57} for any real affine variety, and then continue~$V_K$ separately to the complexifications~$S_{\C}$ of all strata~$S$.

The differential of the volume of a non-degenerate polyhedron $P\colon K\to\Lambda^n$ that is deformed preserving its combinatorial type  is given by Schl\"afli's formula
\begin{equation}\label{eq_Schlafli}
dV_K(P)=-\frac{1}{n-1}\sum_{F}V_F(P)\,d\alpha_F(P),
\end{equation}
where the sum is taken over all $(n-2)$-dimensional simplices $F$ of~$K$, $V_F(P)$ is the $(n-2)$-dimensional volume of the face~$P(F)$, and~$\alpha_F(P)$ is the dihedral angle of~$P$ at the face~$P(F)$. Since every~$V_F$ restricted to~$S$ is a constant, we can integrate equation~\eqref{eq_Schlafli} and obtain that the following equality holds on~$S$:
\begin{equation}\label{eq_integr}
V_K(P)=-\frac{1}{n-1}\sum_{F}V_F(P)\alpha_F(P)+\mathrm{const}.
\end{equation}
It can be checked that the restrictinons to~$S$ of the functions~$Q_F=\exp(i\alpha_F)$  are polynomials, hence, the same polynomials can be considered as functions on~$S_{\C}$. Then the formula
\begin{equation*}
V_K(P)=\frac{i}{n-1}\sum_{F}V_F(P)\Log Q_F(P)+\mathrm{const}
\end{equation*}
yields the analytic continuation of~$V_K$ to a multi-valued analytic function on a Zariski open subset of~$S_{\C}$ such that any two branches of this multi-valued  function differ by a real constant, and the (single-valued) imaginary part of this function has a not more than logarithmic growth. The hardest part of our proof of Theorem~\ref{theorem_BC_Lob} is the proof of the fact that the real part of the obtained analytic function is also single-valued, which is based on the following theorem. 

\begin{theorem}[\cite{Gai15b}]
Suppose that $n$ is odd. Let $V_{\Lambda^n}(G)$ be the function expressing the volume of a bounded simplex in~$\Lambda^n$ from the Gram matrix~$G$ of its vertices. Let~$G_0$ be the Gram matrix of vertices of a non-degenerate simplex in~$\Lambda^n$, and let~$\gamma$ be a closed path in the space of symmetric complex $(n+1)\times(n+1)$ matrices with units on the diagonal such that both endpoints of~$\gamma$ coincide with~$G_0$. Assume that the function $V_{\Lambda^n}(G)$ admits the analytic continuation along~$\gamma$, and let $V_{\Lambda^n}'(G)$ be the holomorphic function in a neighborhood of~$G^0$ obtained after this analytic continuation. Then $\Re V_{\Lambda^n}'(G^0)=\pm V_{\Lambda^n}(G^0)$.
\end{theorem}

A non-constant single-valued holomorphic function on a complex affine algebraic variety cannot have a not more than  logarithmic growth of the imaginary part. Therefore the analytic continuation of~$V_K$ to~$S_{\C}$ is  constant, hence, $V_K$ is constant on~$S$. Thus it is constant on every connected component of~$\Sigma$.

\end{document}